\documentclass[12pt,twoside]{amsart}
\usepackage{amsmath, amsthm, amscd, amsfonts, amssymb, graphicx, bbm, marvosym, mathrsfs, dsfont}
\newtheorem{thm}{Theorem}[section]
\newtheorem{cor}[thm]{Corollary}

\newtheorem{defn}[thm]{Definition}

\numberwithin{equation}{section}

\newtheorem{ill}[thm]{Example}

\begin{document}
	\begin{center}
		{\bf{An application to a system of  $(k,\rho)$-fractional Hilfer integral equations via a measure of noncompactness
		}}
	     \vspace{.5cm}
	     \\ Gete Umbrey$^{a}$, Riken Kaye$^{a}$, Drema Lhamu$^{a}$, Monkhum Khilak$^{a}$,     and Bhuban Chandra Deuri$^{a\ast}$
	     \vspace{.2cm}
	     
	     $^{a}$Department of Mathematics, Jawaharlal Nehru College Pasighat, East Siang-791103, Arunachal Pradesh, India\\
	     %$^{b}$Department of Mathematics, Rajiv Gandhi University, Rono Hills, Doimukh-791112, Arunachal Pradesh, India\\
	     
	     $^{a,\ast}$Corresponding author. E-mail: bhuban.deuri@rgu.ac.in; bhuban.math@gmail.com\\
	     Contributing author. Email: 
	     geteumbrey@gmail.com;
	     riken.kaye@rgu.ac.in;   dremalhamu114@gmail.com  and  monkhum.khilak@rgu.ac.in\
	\end{center}
	
	\title{}
	\author{}
\begin{abstract} 
	In our study, Darbo's fixed point theorem(DFPT) has been extended and generalized using $\mathbb{H}$-class mappings and the measure of noncompactness. Utilizing this Darbo-type theorem, we provided a solvability result for a system of a $(k,\rho)$-fractional Hilfer integral equations, accompanied by an appropriate example to illustrate the findings.
	
	\vskip 0.5cm
	

	\textbf{MSC subject classification No: 26A33; 47H08; 47H09; 47H10}   
\end{abstract}
	
	\maketitle
	\pagestyle{myheadings}
	
	\markboth{\rightline {\scriptsize Brook}}
	{\leftline{\scriptsize }}
	
	\maketitle

	\section{ Introduction}
	Kuratowski's measure of noncompactness (MNC), introduced in 1930 \cite{od2}, has been extensively studied in various mathematical contexts. Different research papers have explored the application of this measure in different areas. Metwali et al. defined a novel MNC in variable exponent Lebesgue spaces, extending classical Lebesgue spaces and enabling the study of various equations  \cite{tro1}. Khokhar et al. proposed the concept of $\mathcal{A}$-condensing operators by utilizing the MNC, focusing on best proximity points and fractional differential equations \cite{tro2}. Telli et al. utilized the Kuratowski MNC in studying boundary-value problems for fractional differential equations with variable order and delays, showcasing the applicability of the measure in stability criteria  \cite{tro3}. 
	
	Caponetti et al.\cite{tro4} analyzed the Kuratowski measure of noncompactness(MNC) in spaces of vector-valued functions, leading to new criteria for compactness based on quantitative features. The MNC is crucial in various fixed point theorems, which are instrumental in exploring solutions to different equations. For instance, the MNC helps in identifying solutions for nonlinear fractional integral equations within Banach spaces, as evidenced by Golshan \cite{tro8}, Deb et al. \cite{tro9}, Metwali \cite{tro1}, and Gabeleh et al.\cite{tro11}. Their research demonstrates the use of MNC in Lebesgue spaces and sequence spaces to confirm the existence of solutions via fixed point theorems. Moreover, combining the MNC with other criteria, like the Hausdorff measure, enhances the effectiveness of fixed point theorems in analyzing the solvability of integral equations, highlighting the synergy between fixed point theorems and the MNC.

	G. Darbo\cite{od5} extended Schauder's fixed point theorem by integrating Kuratowski's MNC into his framework. The MNC has proven vital in broadening Darbo's theorem across various mathematical contexts. Research by Hammad et al. \cite{tro5} introduced a new fixed point theorem that expands Darbo's theorem with the MNC, while Taoudi \cite{tro6} generalized Darbo's principle using a monotone MNC with specific properties. Additionally, Salem et al. \cite{tro7} applied the MNC within Banach spaces to address the existence of solutions for fractional Sturm–Liouville operators using fixed point theorems, including Darbo's. Collectively, these studies underscore the MNC's importance in extending DFPT and solving diverse mathematical problems.
	
	The DFPT, a fundamental mathematical concept, has been extensively studied and generalized. It establishes the existence of fixed points for certain mappings in various spaces. Extensions and generalizations of Darbo's theorem include proofs of solution existence and stability for fractional integral equations \cite{tro13, tro12}, new generalizations with relaxed assumptions, and applications in Banach spaces with monotone MNC and convex mappings \cite{tro14, tro6}. Additionally, new contraction types, such as the F-Darbo contraction, have been proposed, offering broader results and applications for solving integral equations  \cite{tro15}.

	J. Bana\'{s} \cite{Banas1} introduced the MNC in 1980, and its study in Banach spaces has since become a key research area. These measures play a crucial role in determining the existence of solutions to integral boundary value problems \cite{tro16}, studying the representation of an MNC and their applications in Banach spaces \cite{tro17}, exploring interpolation of an MNC of polynomials on Banach spaces \cite{tro18}, examining solvability conditions for fractional integral equations in Banach spaces using the theory of an MNC \cite{tro19}, and applying the MNC in characterizing classes of compact operators, solving integral equations, and establishing the existence of optimal solutions for systems of integro-differentials in Banach spaces \cite{tro11}. These research efforts collectively contribute to a deeper understanding of MNCs and their diverse applications in various mathematical contexts.
	
	Inspired and motivated by \cite{Moti1, Moti2, Moti3, Moti4}, in the context of an MNC. In this paper, we generalized well-known theorems, namely, DFPT and noncompactness measures used in Banach space. Herein, we presented a solvability result of a system of $(k,\rho)$-fractional Hilfer integral equations and a suitable example using the Darbo type theorem.

	\par Let  $\left( \mathscr{B}, \parallel . \parallel \right) $ be a real Banach space and consider the closed ball
	$\mathtt{B} (\omega,r)=\left\lbrace \rho \in \mathscr{B}:\parallel \rho-\omega \parallel \leq r \right\rbrace.$ If $\mathscr{E}(\neq \phi) \subseteq \mathscr{B}.$ we define:
	
	\begin{itemize}
		\item
		$\Re=(-\infty, \infty),$
		\item $\Re_{+}=\left[ 0, \infty \right) ,$
		\item $\bar{\mathscr{E}}$ as the closure of $\mathscr{E},$
		\item $Conv{\mathscr{E}}$ as the convex hull of $\mathscr{E},$
		\item $\mathscr{M}_{\mathscr{B}}$ as the collection of all non-void, bounded subsets of $\mathscr{B},$
		\\	also
		\item $\mathscr{N}_{\mathscr{B}}$ as the collection of all relatively compact sets.
	\end{itemize}

	We provide the below definition of MNC, which is referenced in \cite{Banas1}.
	\begin{defn}\label{def}
		A mapping $\mathscr{Q} :\mathscr{M}_{\mathscr{B}} \rightarrow \Re_{+}$ is  called a MNC in the Banach space $\mathscr{B}$ if it satisfies the following axioms:
		\begin{enumerate}
			\item[(i)] For any $\mathscr{E} \in \mathscr{M}_{\mathscr{B}},$ if
			$\mathscr{Q}(\mathscr{E})=0$ gives $\mathscr{E}$ is relatively compact.
			\item[(ii)]  ker $\mathscr{Q} = \left\lbrace \mathscr{E} \in \mathscr{M}_{\mathscr{B}}: \mathscr{Q}\left( \mathscr{E}\right)=0   \right\rbrace$ which is non-empty and $ \mathscr{N}_{\mathscr{B}}\supset$ ker $\mathscr{Q}.$ 
			\item[(iii)] $\mathscr{E} \subseteq \mathscr{E}_{1} \implies \mathscr{Q}\left( \mathscr{E}\right) \leq \mathscr{Q}\left( \mathscr{E}_{1}\right).$
			\item[(iv)] $\mathscr{Q}\left( \bar{\mathscr{E}}\right)=\mathscr{Q}\left( \mathscr{E}\right).$
			\item[(v)] $\mathscr{Q}\left( Conv \mathscr{E}\right)=\mathscr{Q}\left( \mathscr{E}\right).$
			\item[(vi)] For any $\mathscr{L} \in \left[ 0, 1 \right],$ we have $\mathscr{Q}\left( \mathscr{L} \mathscr{E} +\left(1- \mathscr{L} \right)\mathscr{E}_{1} \right) \leq \mathscr{L} \mathscr{Q}\left( \mathscr{E}\right)+\left(1- \mathscr{L} \right)\mathscr{Q}\left( \mathscr{E}_{1}\right).$ 
			\item[(vii)] If $\mathscr{E}_{k} \in \mathscr{M}_{\mathscr{B}}\; are\; closed \:(\mathscr{E}_{k}= \bar{\mathscr{E}}_{k})\; and\;	\mathscr{E}_{k+1} \subset \mathscr{E}_{k}$ for $k=1,2,3,\dots$ and if $\lim\limits_{k \rightarrow \infty}\mathscr{Q}\left( \mathscr{E}_{k}\right)=0,$ then $\bigcap_{k=1}^{\infty}\mathscr{E}_{k} \neq \emptyset.$ 
		\end{enumerate}
	\end{defn}
	\par The \textit{kernel of measure} $\mathscr{Q}$ is defined as the family $ker \mathscr{Q} .$ Since $\mathscr{Q}(\mathscr{E}_{\infty}) \leq \mathscr{Q}(\mathscr{E}_{k})$ for all $k,$ it follows that $\mathscr{Q}(\mathscr{E}_{\infty})=0,$  $\Rightarrow\mathscr{E}_{\infty}=\bigcap_{k=1}^{\infty}\mathscr{E}_{k} \in ker \mathscr{Q}.$

	\subsection*{Some key theorems and definitions}
	\par We recall the following key theorems:
	\begin{thm}(\;$\mathtt{Shauder}$,\;\cite{sh}\; )\label{s1}
		Suppose that $\mathscr{P}$ be a nonempty, bounded,closed and convex subset$(\mathcal{NBCCS})$ of a Banach space $\mathscr{Y}$ and, the mapping $\mathscr{F}: \mathscr{P} \rightarrow \mathscr{P}$ is a continuous  with compact. So, $\mathscr{F}$ has atleast one fixed point.
		%Let\; $\mathscr{P}$\; be\; a\; non-empty,\; bounded,\; closed\; and \;convex\; subset\;$(\mathcal{NBCCS})$ \;of\; a\; Banach\; Space\; $\mathscr{Y}.$ If\;  $\mathscr{F}: \mathscr{P} \rightarrow \mathscr{P}$ \;is\; a\; continuous\; and\; compact\; mapping,\; then\; it\; must\; have\; at\; least\; one\; fixed\; point.
	\end{thm}
	\begin{thm}(\;$\mathtt{Darbo}$,\;\cite{od5}\;)\label{l1}
		Suppose that $\mathscr{P}$ be a $\mathcal{NBCCS}$ of the Banach space $\mathscr{Y}$ and, the self mapping $\mathscr{F}: \mathscr{P} \rightarrow \mathscr{P}$ is a continuous  such that
		%Let \; $\mathscr{P}$\; be\; a\; $\mathcal{NBCCS}$\; of\; a\; Banach\; Space\; $\mathscr{Y}.$ Assume $\mathscr{F}: \mathscr{P} \rightarrow \mathscr{P}$ \;is\; a\; continuous\; self\; mapping\; with 
		\[
		\mathscr{Q}(\mathscr{FC})\leq m\mathscr{Q}(\mathscr{C}),\; \mathscr{C}\subseteq \mathscr{P}.
		\]
		$Where$ $m\in \left[ 0,1\right) $,
		so the mapping $\mathscr{F}$ have  a fixed point.
	\end{thm}

	\par In order to show an expansion of Darbo's FPT, the following related definitions are required :
	%With the help of following concepts, we establish our fixed point theorem.\\
	
	Let $\Gamma$ be the class of all functions $\gamma : \Re_{+} \rightarrow \Re_{+},$ so $\gamma(0)=0$ and $\gamma(\theta)>0$ for $\theta>0.$
	so that 
	\begin{defn}(\cite{das1})\label{d2}
		Denote by $\Upsilon$ the class of all $\upsilon : \Re_{+}\rightarrow \Re_{+},$ so 
		\begin{enumerate}
			\item $\upsilon(\sigma) = 0 \Leftrightarrow \sigma = 0.$
			\item $\upsilon$ is non-decreasing and continuous.
		\end{enumerate}
	\end{defn}
	\begin{defn}(\cite{das1})\label{d3}
		Denote by $\mathbb{H}$ the class of all  $\mathscr{H}: \Re_{+}^{2}  \rightarrow \Re_{+}, $ so satisfying :
		\begin{enumerate}
			\item $\mathscr{H}(z_1,z_2)\geq \max \left\lbrace z_1,z_2\right\rbrace $ for $z_1,z_2 \geq 0.$
			\item $\mathscr{H}$ is a continuous
			\item $\mathscr{H}(z_1+z_2,x_1+x_2) \leq \mathscr{H}(z_1,x_1)+\mathscr{H}(z_2,x_2).$
		\end{enumerate}
		e.g $\mathscr{H}(z_1,z_2) = z_1+z_2.$
	\end{defn}

	%**************************** Main Result*********************
	\section{The Main Result (FPT)}
	\begin{thm}\label{th1}
		Assume that $\mathscr{M}$ is a $\mathcal{NBCCS}$ of the Banach space $\mathscr{Y}$ and  $\mathscr{F}:\mathscr{M} \rightarrow \mathscr{M}$ is a continuous function such that
		%Let\; $\mathscr{M}$\; be\; a\; $\mathcal{NBCCS}$\; of\; a\; Banach\; space\; $\mathscr{Y}.$ Also,\; the\; function\; $\mathscr{F}:\mathscr{M} \rightarrow \mathscr{M}$\; is\; continuous \;with
		\begin{equation}\label{eq1}
			\upsilon[ \mathscr{H}\left(\mathscr{Q}( \mathscr{F}\mathscr{T}),\varphi(\mathscr{Q}( \mathscr{F}\mathscr{T}))\right)] \leq 
			\upsilon[ \mathscr{H}\left(\mathscr{Q}( \mathscr{T}),\varphi(\mathscr{Q}( \mathscr{T}))\right)]
			-\gamma[ \mathscr{H}\left(\mathscr{Q}( \mathscr{T}),\varphi(\mathscr{Q}( \mathscr{T}))\right)]
		\end{equation}
		where $\mathscr{T} \subset \mathscr{M}$ and $\mathscr{Q}$ is an arbitrary $MNC$  and $ \upsilon \in \Upsilon, \;\gamma \in \Gamma,\; \mathscr{H} \in \mathbb{H}\;$  also a function $ \varphi : \Re_{+} \rightarrow \Re_{+}$ is nondecreasing continuous .
		So that $\mathscr{F}$ has atleast one fixed point in $\mathscr{M}.$
	\end{thm}
	\begin{proof}
		Assume that a sequence $\left\lbrace \mathscr{A}_{p}\right\rbrace _{p=1}^{\infty}$ via $\mathscr{A}_{0}=\mathscr{A}$
		and $\mathscr{A}_{p}=Conv(\mathscr{F}\mathscr{A}_{p-1})$ for $p \in \mathbb{N}.$  
		%By continuing in the similar manner gives
		Proceeding in the same way gives 
		$\mathscr{A}_{0} \supseteq\mathscr{A}_{1} \supseteq \mathscr{A}_{2}\supseteq \mathscr{A}_{3}\supseteq \ldots \supseteq \mathscr{A}_{p}\supseteq \mathscr{A}_{p+1}\supseteq \ldots.$

		\par If there exists a $\mathscr{N}\in \mathbb{N}$ satisfying $\mathscr{H}\left(\mathscr{Q}( \mathscr{A}_{\mathscr{N}}),\varphi(\mathscr{Q}( \mathscr{A}_{\mathscr{N}}))\right) =0.$ which gives
		$\mathscr{Q}( \mathscr{A}_{\mathscr{N}})=0$ , so $\mathscr{A}_{\mathscr{N}}$ is a compact set. Then, by using Schauder's  theorem we obtain that $\mathscr{F}$ has a fixed point in $\mathscr{A}.$

		\par Again, if $\mathscr{H}\left(\mathscr{Q}( \mathscr{A}_{p}),\varphi(\mathscr{Q}( \mathscr{A}_{p}))\right)>0,\;\; p \in \mathbb{N}.$

		\par Now, for  $\; p \in \mathbb{N},$  
		\begin{align*}
			& \upsilon[ \mathscr{H}\left(\mathscr{Q}( \mathscr{A}_{p+1}),\varphi(\mathscr{Q}( \mathscr{A}_{p+1}))\right)]\\
			& = \upsilon[ \mathscr{H}\left(\mathscr{Q}( Conv\mathscr{F}\mathscr{A}_{p}),\varphi(\mathscr{Q}(Conv\mathscr{F} \mathscr{A}_{p}))\right)]\\
			&= \upsilon[ \mathscr{H}\left(\mathscr{Q}( \mathscr{F}\mathscr{A}_{p}),\varphi(\mathscr{Q}(\mathscr{F} \mathscr{A}_{p}))\right)]\\
			& \leq \upsilon[ \mathscr{H}\left(\mathscr{Q}( \mathscr{A}_{p}),\varphi(\mathscr{Q}( \mathscr{A}_{p}))\right)]- \gamma[ \mathscr{H}\left(\mathscr{Q}( \mathscr{A}_{p}),\varphi(\mathscr{Q}( \mathscr{A}_{p}))\right)]\\
			& < \upsilon[ \mathscr{H}\left(\mathscr{Q}( \mathscr{A}_{p}),\varphi(\mathscr{Q}( \mathscr{A}_{p}))\right)].
		\end{align*}
		Since, $\left\lbrace \upsilon[ \mathscr{H}\left(\mathscr{Q}( \mathscr{A}_{p}),\varphi(\mathscr{Q}( \mathscr{A}_{p}))\right)]\right\rbrace _{p=1}^{\infty}  $ is a positive decreasing sequence  that is bounded below.\\
		So, $\lim\limits_{p \rightarrow \infty} \upsilon[ \mathscr{H}\left(\mathscr{Q}( \mathscr{A}_{p}),\varphi(\mathscr{Q}( \mathscr{A}_{p}))\right)]=\varpi\geqslant 0.$

		If possible let $\varpi>0.$ As $p \rightarrow \infty$, we have
		\[\varpi<\varpi.\]
		Which is a contradictions.\\
		So,
		$\upsilon[ \mathscr{H}\left(\mathscr{Q}( \mathscr{A}_{p}),\varphi(\mathscr{Q}( \mathscr{A}_{p}))\right)] \rightarrow 0 $ as $p \rightarrow \infty.$
		\\ This implies that\\
		$ \mathscr{H}\left(\mathscr{Q}( \mathscr{A}_{p}),\varphi(\mathscr{Q}( \mathscr{A}_{p}))\right) \rightarrow 0 $ as $p \rightarrow \infty.$\\
		Using the property of $\mathscr{H} $, we get\\ 
		$ \mathscr{Q}( \mathscr{A}_{p}) \rightarrow 0 $ as $p \rightarrow \infty.$

		\par Since $\mathscr{A}_{p}\supseteq \mathscr{A}_{p+1},$ with definition\; \ref{def} concluding that $\mathscr{A}_{\infty}=\bigcap_{p=1}^{\infty}\mathscr{A}_{p}$ is nonempty, closed and convex subset of $\mathscr{A}$ and $\mathscr{A}_{\infty}$ is invariant under $\mathscr{F}$ .
		\\ Thus, Schauder’s FPT (theorem \;\ref{s1})\;concludes\;that $\mathscr{F}$\;has\;a\;fixed\;point\;in\;$\mathscr{A}.$
		Thus, the completed proof.
	\end{proof}

	%**************************** Corollary & Theorem*********************
	\begin{thm}\label{th2}
		Assume that $\mathscr{M}$ is a $\mathcal{NBCCS}$ of the Banach space $\mathscr{Y}$ and  $\mathscr{F}:\mathscr{M} \rightarrow \mathscr{M}$ is a continuous function  such that
		%Let\; $\mathscr{M}$\; be\; a\; $\mathcal{NBCCS}$\; of\; a\; Banach\; space\; $\mathscr{Y}.$ Also,\; the\; function\; $\mathscr{F}:\mathscr{M} \rightarrow \mathscr{M}$\; is \;continuous \;with
		\begin{equation}\label{con2}
			\mathscr{H}\left(\mathscr{Q}( \mathscr{F}\mathscr{T}),\varphi(\mathscr{Q}( \mathscr{F}\mathscr{T}))\right) \leq 
			(1-2\mathscr{G})\mathscr{H}\left(\mathscr{Q}( \mathscr{T}),\varphi(\mathscr{Q}( \mathscr{T}))\right)
		\end{equation}
		where $\mathscr{T} \subset \mathscr{M}$ and $\mathscr{Q}$ is an arbitrary $MNC$  and $ \mathscr{H} \in \mathbb{H}\;$  also a function $ \varphi : \Re_{+} \rightarrow \Re_{+}$ is nondecreasing continuous.
		So that $\mathscr{F}$ has atleast one fixed point in $\mathscr{M}.$
	\end{thm}
	\begin{proof}
		Setting $\upsilon(x)=\frac{x}{2}$ and $\gamma(x)=\mathscr{G}x, \; 0<\mathscr{G}<1$ in the Theorem \; \ref{th1}. This yields Theorem \ref{th2}.
	\end{proof}
	\begin{thm}\label{th3}
		Assume that $\mathscr{M}$ is a $\mathcal{NBCCS}$ of the Banach space $\mathscr{Y}$ and  $\mathscr{F}:\mathscr{M} \rightarrow \mathscr{M}$ is a continuous function  such that
		\begin{equation}\label{con3}
			\mathscr{Q}( \mathscr{F}\mathscr{T})+\varphi(\mathscr{Q}( \mathscr{F}\mathscr{T})) \leq 
			(1-2\mathscr{G})\left(\mathscr{Q}( \mathscr{T})+\varphi(\mathscr{Q}( \mathscr{T}))\right)
		\end{equation}
		where $\mathscr{T} \subset \mathscr{M}$ and $\mathscr{Q}$ is an arbitrary $MNC$  also  a  function $ \varphi : \Re_{+} \rightarrow \Re_{+}$ is nondecreasing continuous.
		So that $\mathscr{F}$ has atleast one fixed point in $\mathscr{M}.$
	\end{thm}
	\begin{proof}
		Setting $\mathscr{H}(z_1,z_2) = z_1+z_2$ in the Theorem\; \ref{th2}. This yields Theorem \ref{con3}.
	\end{proof}
	
	\begin{cor}\label{c1}
		Assume that $\mathscr{M}$ be a $\mathcal{NBCCS}$ of the Banach space $\mathscr{Y}$ and $\mathscr{F}:\mathscr{M} \rightarrow \mathscr{M}$ is a continuous function such that
		\begin{equation}\label{con4}
			\mathscr{Q}( \mathscr{F}\mathscr{T}) \leq 
			m\mathscr{Q}( \mathscr{T})
		\end{equation}
		where $\mathscr{T} \subset \mathscr{M}$ and $\mathscr{Q}$ is an arbitrary $MNC.$
		So that $\mathscr{F}$ has atleast one fixed point in $\mathscr{M}.$
	\end{cor}
	\begin{proof}
		Setting $\varphi(x) =0, \; 0<m<1$ in the Theorem\; \ref{th3}. We obtain, DFPT.
	\end{proof}

	%--------------------------------------------------	
	%---------------------------------------------------	
	
		\section{Application}
	%**************************** Measure of noncompactness *********************
	\subsection{MNC on $\Bar{\mathds{C}}([0,\mathbb{T}])$:}
	Assume that  $\mathscr{B}=\Bar{\mathds{C}}(\Bar{\mathds{J}})=\Bar{\mathds{C}}([0,\mathbb{T}])$ is the space of all real valued continuous function defined on $\Bar{\mathds{J}}.$ Then,
	\[
	\parallel \mathscr{T} \parallel=\sup\left\lbrace \left| \mathscr{T}(z)\right|:z \in \Bar{\mathds{J}} \right\rbrace ,\; \mathscr{T}\in \mathscr{B}.
	\]
	
	Suppose $\mathscr{F}(\neq \phi) \subseteq \mathscr{B}$ is   bounded.\\ Fix  $\delta >0$ and $\mathscr{T} \in \mathscr{F}$ such that modulus of the continuity of $\mathscr{T}$ is defined as
	\[\mu(\mathscr{T},\delta)=\sup \left\lbrace \left| \mathscr{T}(z_{1})-\mathscr{T}(z_{2})\right|: z_{1},z_{2} \in \Bar{\mathds{J}} , \left| z_{1}-z_{2}\right|\leq \delta \right\rbrace.
	\]
	Moreover, we set
	\[ \mu(\mathscr{F},\delta)=\sup\left\lbrace \mu(\mathscr{T},\delta): \mathscr{T} \in \mathscr{F}\right\rbrace;~ \mu_{0}(\mathscr{F})=\lim\limits_{\delta \rightarrow 0}\mu(\mathscr{F},\delta).
	\]
	Here, the mapping $\mu_{0}$ is known as MNC in $\mathscr{B}$, with $\mathscr{Q}(\mathscr{F})=\frac{1}{2}\mu_{0}(\mathscr{F})$
	(see \cite{Banas1}) as the Hausdorff MNC $\mathscr{Q}$.

	%**************************** Solvability fractional integral equation *********************
	\subsection{Existence results of a fractional integral equations:}
	%%%%%%%%%%%%%%%%%%%%%%%%%%%%%%%%%%%
	Let $-\infty\leq a<b<\infty,$ $\;0<\rho,k,\gamma<1$ and also let $\phi$ be a continuous function. Then, defines the integral $(k,\rho)$-fractional Hilfer left integral is given by (we reffer \cite{Cap2,Cap3,Cap4}).
	\[
	(^{\rho}_{k}\mathscr{J}^{\gamma}_{a+}\phi)(x)=\frac{\rho^{1-\frac{\gamma}{k}}}{k\Gamma_{k}(\gamma)}\int_{a}^{x}\frac{t^{\rho-1}}{\left(x^{\rho}-t^{\rho} \right)^{1-\frac{\gamma}{k}}}\phi(t)dt ,\;\;\;x>a\;\;x \in [a,b].
	\]
	
	Given that $a=1$ and $b=\mathbb{T}$, we compute
	\[
	(^{\rho}_{k}\mathscr{J}^{\gamma}\phi)(x)=\frac{\rho^{1-\frac{\gamma}{k}}}{k\Gamma_{k}(\gamma)}\int_{1}^{x}\frac{t^{\rho-1}}{\left(x^{\rho}-t^{\rho} \right)^{1-\frac{\gamma}{k}}}\phi(t)dt ,\;\;\;x>a\;\;x \in [1,\mathbb{T}].
	\]

	In this study, we examine the system of fractional integral equations:
	
	\begin{equation}\label{eq75}
		\alpha(x)= F(x,\alpha(x))+\frac{\rho^{1-\frac{\gamma}{k}}\Psi(x,\alpha(x))}{k\Gamma_{k}(\gamma)}\int_{1}^{x}\frac{t^{\rho-1}}{\left(x^{\rho}-t^{\rho} \right)^{1-\frac{\gamma}{k}}}G(t,\alpha(t))dt ,\;\;\;x>a\;\;x \in [1,\mathbb{T}]
	\end{equation}
	and
	\begin{equation}\label{eq77}
		\beta(x)= \hat{F}(x,\beta(x))+\frac{\rho^{1-\frac{\gamma}{k}}\hat{\Psi}(x,\beta(x))}{k\Gamma_{k}(\gamma)}\int_{1}^{x}\frac{t^{\rho-1}}{\left(x^{\rho}-t^{\rho} \right)^{1-\frac{\gamma}{k}}}\hat{G}(t,\beta(t))dt ,\;\;\;x>a\;\;x \in [1,\mathbb{T}]
	\end{equation}
	where $x \in \Bar{\mathds{J}}=[1,\mathbb{T}]\;\;and\;\;0< \rho,k,\gamma<1.$\\
	Let  define 
	\[
	\mathscr{Q}_{r_{0}}=\left\lbrace \alpha\in \mathscr{B}: \parallel \alpha \parallel \leq r_{0} \right\rbrace \].
	Assume that
	\begin{enumerate}
		\item [(C1)]
		Let $F,\; \hat{F},\; \Psi,\; \hat{\Psi},\; G,\; \hat{G}:\Bar{\mathds{J}} \times \Re^{+} \rightarrow \Re$ be continuous and there exist constants $c_{1},\; \hat{c}_{1},\; c_{2},\; \hat{c}_{2},\; c_{3},\; \hat{c}_{3} \geq 0$ satisfying
		\[
		\left| F(x,\alpha(x))-F(x,\bar{\alpha}(x))\right|\leq c_{1} \left| \alpha(x)-\bar{\alpha}(x)\right|
		\]
		\[
		\left| \hat{F}(x,\beta(x))-\hat{F}(x,\bar{\beta}(x))\right|\leq \hat{c}_{1} \left| \beta(x)-\bar{\beta}(x)\right|
		\]
		\[
		\left| \Psi(x,\alpha(x))-\Psi(x,\bar{\alpha}(x))\right|\leq c_{2} \left| \alpha(x)-\bar{\alpha}(x)\right|
		\]
		\[
		\left| \hat{\Psi}(x,\beta(x))-\hat{\Psi}(x,\bar{\beta}(x))\right|\leq \hat{c}_{2} \left| \beta(x)-\bar{\beta}(x)\right|
		\]
		\[
		\left| G(x,\alpha(x))-G(x,\bar{\alpha}(x)) \right|\leq c_{3} \left| \alpha(x)-\bar{\alpha}(x)\right|
		\]
		and
		\[
		\left|  \hat{G}(x,\beta(x))-\hat{G}(x,\bar{\beta}(x)) \right|\leq \hat{c}_{3} \left| \beta(x)-\bar{\beta}(x)\right|
		\]
		Also, for all $x \in \Re^{+},$
		\[
		F(x,0)=0=\hat{F}(x,0)
		\]
		\[
		\Psi(x,0)=0=\hat{\Psi}(x,0)
		\]
		and
		\[
		G(x,0)=0=\hat{G}(x,0)
		\]
		\item[(C2)] The given below inequality has a positive result $r_{0}$ such that
		%$\exists$ a positive solution $r_{0}$ of the inequality
		\begin{equation*}
			c_{1}+
			\frac{c_{2}c_{3}  r_{0}\rho^{-\frac{\gamma}{k}}}{\gamma \Gamma_{k}(\gamma)}\left( \mathbb{T}^{\rho}-1\right)^{\frac{\gamma}{k}}< 1
		\end{equation*}
		and 
		\begin{equation*}
			\hat{c}_{1}+
			\frac{\hat{c}_{2}\hat{c}_{3}  r_{0}\rho^{-\frac{\gamma}{k}}}{\gamma \Gamma_{k}(\gamma)}\left( \mathbb{T}^{\rho}-1\right)^{\frac{\gamma}{k}}< 1.
		\end{equation*}
	\end{enumerate}

	%**************************** Application Theorem *********************
	\begin{thm}\label{T74}
		If the assumptions $(C1)-(C2)$ holds, then the equation (\ref{eq75}) and (\ref{eq77}) have  solutions in $\mathscr{B}=\Bar{\mathds{C}}(\Bar{\mathds{J}}).$
		%If conditions  $(H1)-(H4)$ satisfied, so the eq.(\ref{eq75}) has atleast a solution in $\mathscr{B}=\Bar{\mathds{C}}(\Bar{\mathds{J}}).$
	\end{thm}
	\begin{proof}
	We assume the operator $\mathscr{D}, \mathscr{\bar{D}}: \mathscr{B} \rightarrow \mathscr{B}$  defined by:
		\[(\mathscr{D} \alpha)(x)= F(x,\alpha(x))+\frac{\rho^{1-\frac{\gamma}{k}}\Psi(x,\alpha(x))}{k\Gamma_{k}(\gamma)}\int_{1}^{x}\frac{t^{\rho-1}}{\left(x^{\rho}-t^{\rho} \right)^{1-\frac{\gamma}{k}}}G(t,\alpha(t))dt
		\]
		and
		\[(\mathscr{\bar{D}} \beta)(x)= \hat{F}(x,\beta(x))+\frac{\rho^{1-\frac{\gamma}{k}}\hat{\Psi}(x,\beta(x))}{k\Gamma_{k}(\gamma)}\int_{1}^{x}\frac{t^{\rho-1}}{\left(x^{\rho}-t^{\rho} \right)^{1-\frac{\gamma}{k}}}\hat{G}(t,\beta(t))dt
		\]

		%**************************** Step1 *********************	
		{\bf Step 1:}   We claim that $\mathscr{D}, \mathscr{\bar{D}}$
		maps $\mathscr{Q}_{r_0}$ into $\mathscr{Q}_{r_0}$.  Assume that  $\alpha,\; \beta \in \mathscr{Q}_{r_{0}}$. We now have
		\begin{eqnarray*}
			\left| (\mathscr{D} \alpha)(x)\right|&=&\left|F(x,\alpha(x))+\frac{\rho^{1-\frac{\gamma}{k}}\Psi(x,\alpha(x))}{k\Gamma_{k}(\gamma)}\int_{1}^{x}\frac{t^{\rho-1}}{\left(x^{\rho}-t^{\rho} \right)^{1-\frac{\gamma}{k}}}G(t,\alpha(t))dt\right|\\
			&\leq&\left|F(x,\alpha(x))\right|+\frac{\rho^{1-\frac{\gamma}{k}}\left|\Psi(x,\alpha(x))\right|}{k\Gamma_{k}(\gamma)}\int_{1}^{x}\frac{t^{\rho-1}}{\left(x^{\rho}-t^{\rho} \right)^{1-\frac{\gamma}{k}}}\left|G(t,\alpha(t))\right|dt\\
			&\leq& \left|F(x,\alpha(x))-F(x,0)\right|\\
			&+&\frac{\rho^{1-\frac{\gamma}{k}}\left\lbrace \left|\Psi(x,\alpha(x))-\Psi(x,0)\right|+\left|\Psi(x,0)\right|\right\rbrace }{k\Gamma_{k}(\gamma)}\\
			&&\int_{1}^{x}\frac{t^{\rho-1}}{\left(x^{\rho}-t^{\rho} \right)^{1-\frac{\gamma}{k}}}\left\lbrace \left|G(t,\alpha(t))-G(t,0)\right|+\left|G(t,0)\right|\right\rbrace dt\\
			&\leq&c_{1}\left|\alpha(x)\right|+\frac{\rho^{1-\frac{\gamma}{k}}c_{2} \left|\alpha(x)\right|}{k\Gamma_{k}(\gamma)}\int_{1}^{x}\frac{t^{\rho-1}}{\left(x^{\rho}-t^{\rho} \right)^{1-\frac{\gamma}{k}}}c_{3} \left|\alpha(t)\right| dt\\
			&\leq&c_{1}\parallel \alpha \parallel+\frac{\rho^{1-\frac{\gamma}{k}}c_{2}c_{3}  \parallel \alpha \parallel^{2}}{k\Gamma_{k}(\gamma)}\int_{1}^{x}\frac{t^{\rho-1}}{\left(x^{\rho}-t^{\rho} \right)^{1-\frac{\gamma}{k}}} dt\\
			&\leq&c_{1}\parallel \alpha \parallel+
			\frac{c_{2}c_{3}  \parallel \alpha \parallel^{2}\rho^{-\frac{\gamma}{k}}}{\gamma \Gamma_{k}(\gamma)}\left( \mathbb{T}^{\rho}-1\right)^{\frac{\gamma}{k}}.
		\end{eqnarray*}
		Hence, $\parallel \alpha  \parallel <r_0$ gives
		\[
		\parallel \mathscr{D} \alpha\parallel\leq  c_{1}r_{0}+
		\frac{c_{2}c_{3}  r_{0}^{2}\rho^{-\frac{\gamma}{k}}}{\gamma \Gamma_{k}(\gamma)}\left( \mathbb{T}^{\rho}-1\right)^{\frac{\gamma}{k}}\leq r_{0}
		\]
		It folloews from $(C2)$ that $\mathscr{D}$
		maps $\mathscr{Q}_{r_0}$ into $\mathscr{Q}_{r_0}$.\\
		In the similar manner it can be established that $\mathscr{D}$ maps from $\mathscr{Q}_{r_0}$ to itself.

		%**************************** Step2 *********************	
		{\bf Step 2:} Next we claim that, $\mathscr{D}$ is continuous on $\mathscr{Q}_{r_{0}}.$

		Setting $\delta >0,$ also let $\alpha, \bar{\alpha} \in \mathscr{Q}_{r_{0}}$ such that $\parallel \alpha- \bar{\alpha} \parallel < \delta.$ We now obtain
		\begin{align*}
			&\left|  (\mathscr{D} \alpha)(x)-  (\mathscr{D} \bar{\alpha})(x)\right|\\ &\leq 
			\left|F(x,\alpha(x))-F(x,\bar{\alpha}(x))\right|\\
			&+ \frac{\rho^{1-\frac{\gamma}{k}}}{k\Gamma_{k}(\gamma)}\left|\Psi(x,\alpha(x))\int_{1}^{x}\frac{t^{\rho-1}}{\left(x^{\rho}-t^{\rho} \right)^{1-\frac{\gamma}{k}}}G(t,\alpha(t))dt-\Psi(x,\bar{\alpha}(x))\int_{1}^{x}\frac{t^{\rho-1}}{\left(x^{\rho}-t^{\rho} \right)^{1-\frac{\gamma}{k}}}G(t,\bar{\alpha}(t))dt\right|\\
			&\leq 
			c_{1}\left|\alpha(x)-\bar{\alpha}(x)\right|\\
			&+ \frac{\rho^{1-\frac{\gamma}{k}}}{k\Gamma_{k}(\gamma)}\left|\Psi(x,\alpha(x))\int_{1}^{x}\frac{t^{\rho-1}}{\left(x^{\rho}-t^{\rho} \right)^{1-\frac{\gamma}{k}}}\left\lbrace G(t,\alpha(t))-G(t,\bar{\alpha}(t))\right\rbrace dt\right|\\
				&+ \frac{\rho^{1-\frac{\gamma}{k}}}{k\Gamma_{k}(\gamma)}\left|\left\lbrace \Psi(x,\alpha(x))-\Psi(x,\bar{\alpha}(x))\right\rbrace \int_{1}^{x}\frac{t^{\rho-1}}{\left(x^{\rho}-t^{\rho} \right)^{1-\frac{\gamma}{k}}}G(t,\bar{\alpha}(t))dt\right|\\
			&\leq 
			c_{1}\left|\alpha(x)-\bar{\alpha}(x)\right|+ \frac{\rho^{1-\frac{\gamma}{k}}c_{1}\left|\alpha(x)\right|}{k\Gamma_{k}(\gamma)}\int_{1}^{x}\frac{t^{\rho-1}}{\left(x^{\rho}-t^{\rho} \right)^{1-\frac{\gamma}{k}}}\left| G(t,\alpha(t))-G(t,\bar{\alpha}(t))\right| dt\\
			&+ \frac{\rho^{1-\frac{\gamma}{k}} c_{2}\left|\alpha(x)-\bar{\alpha}(x)\right|}{k\Gamma_{k}(\gamma)} \int_{1}^{x}\frac{t^{\rho-1}}{\left(x^{\rho}-t^{\rho} \right)^{1-\frac{\gamma}{k}}}\left|G(t,\bar{\alpha}(t))\right|dt\\
			&\leq 
			c_{1}\left|\alpha(x)-\bar{\alpha}(x)\right|+ \frac{\rho^{1-\frac{\gamma}{k}}c_{1}\left|\alpha(x)\right|}{k\Gamma_{k}(\gamma)}\int_{1}^{x}\frac{t^{\rho-1}}{\left(x^{\rho}-t^{\rho} \right)^{1-\frac{\gamma}{k}}}c_{3}\left| \alpha(t)-\bar{\alpha}(t)\right| dt\\
			&+ \frac{\rho^{1-\frac{\gamma}{k}} c_{2}\left|\alpha(x)-\bar{\alpha}(x)\right|}{k\Gamma_{k}(\gamma)} \int_{1}^{x}\frac{t^{\rho-1}}{\left(x^{\rho}-t^{\rho} \right)^{1-\frac{\gamma}{k}}}c_{3}\left|\bar{\alpha}(t)\right|dt\\
			&\leq 
			c_{1}\parallel \alpha-\bar{\alpha}\parallel + \frac{\rho^{1-\frac{\gamma}{k}}c_{1}c_{3}\parallel \alpha\parallel \parallel  \alpha-\bar{\alpha}\parallel }{k\Gamma_{k}(\gamma)}\int_{1}^{x}\frac{t^{\rho-1}}{\left(x^{\rho}-t^{\rho} \right)^{1-\frac{\gamma}{k}}}  dt\\
			&+ \frac{\rho^{1-\frac{\gamma}{k}} c_{2}c_{3}\parallel \bar{\alpha}\parallel\parallel \alpha-\bar{\alpha}\parallel }{k\Gamma_{k}(\gamma)} \int_{1}^{x}\frac{t^{\rho-1}}{\left(x^{\rho}-t^{\rho} \right)^{1-\frac{\gamma}{k}}} dt\\			
				&\leq 
			c_{1}\delta + \frac{\rho^{-\frac{\gamma}{k}}c_{1}c_{3}r_{0} \delta }{\gamma\Gamma_{k}(\gamma)}\left( \mathbb{T}^{\rho}-1\right)^{\frac{\gamma}{k}}+ \frac{\rho^{-\frac{\gamma}{k}} c_{2}c_{3}r_{0}\delta }{\gamma\Gamma_{k}(\gamma)} \left( \mathbb{T}^{\rho}-1\right)^{\frac{\gamma}{k}}.
			%& \leq \frac{k}{\rho\gamma}\left( \mathbb{T}^{\rho}-1\right)^{\frac{\gamma}{k}}.
		\end{align*}
		
		As $\delta \rightarrow 0$, we get
		$$\left|  (\mathscr{D} \alpha)(x)-  (\mathscr{D} \bar{\alpha})(x)\right| \rightarrow 0.$$
		\\ which indicates that $\mathscr{D}$ is continuous on $\mathscr{Q}_{r_{0}}.$\\
		In the similar manner it can be established that $\mathscr{D}$ is continuous on $\mathscr{Q}_{r_{0}}.$

		%**************************** Step3 *********************	
		{\bf Step 3:} Here, we estimate $\mathscr{D}$ and $\mathscr{\bar{D}}$ in relation to $\mu_{0}.$ %We here estimate of $\mathscr{D}$ with respect to $\mu_{0}$.
		Consider that $\varpi_{\alpha},  \varpi_{\beta}(\neq \emptyset) \subseteq \mathscr{Q}_{r_{0}}.$  Let  $\delta >0$ be arbitrary, Also, we are now choose $\alpha \in \varpi_{\alpha},$ $\beta \in \varpi_{\beta},$ with $x_{1}, x_{2} \in \Bar{\mathds{J}}$ such that $\left| x_{2}-x_{1}\right|\leq \delta $ and $x_{2} \geq x_{1}.$
		\par Now,
		\begin{align*}
			&\left| (\mathscr{D} \alpha)(x_{2})-  (\mathscr{D} \alpha)(x_{1})\right|\\
			&\leq 
			\left|F(x_{2},\alpha(x_{2}))-F(x_{1},\alpha(x_{1}))\right|\\
			&+\frac{\rho^{1-\frac{\gamma}{k}}}{k\Gamma_{k}(\gamma)}\left|\Psi(x_{2},\alpha(x_{2})) \int_{1}^{x_{2}}\frac{t^{\rho-1}}{\left(x_{2}^{\rho}-t^{\rho} \right)^{1-\frac{\gamma}{k}}}G(t,\alpha(t))dt-\Psi(x_{1},\alpha(x_{1}))\int_{1}^{x_{1}}\frac{t^{\rho-1}}{\left(x_{1}^{\rho}-t^{\rho} \right)^{1-\frac{\gamma}{k}}}G(t,\alpha(t))dt\right|\\
				&\leq 
			\left|F(x_{2},\alpha(x_{2}))-F(x_{1},\alpha(x_{1}))\right|\\
			&+\frac{\rho^{1-\frac{\gamma}{k}}}{k\Gamma_{k}(\gamma)}\left|\Psi(x_{2},\alpha(x_{2})) \left\lbrace \int_{1}^{x_{2}}\frac{t^{\rho-1}}{\left(x_{2}^{\rho}-t^{\rho} \right)^{1-\frac{\gamma}{k}}}G(t,\alpha(t))dt-\int_{1}^{x_{1}}\frac{t^{\rho-1}}{\left(x_{1}^{\rho}-t^{\rho} \right)^{1-\frac{\gamma}{k}}}G(t,\alpha(t))dt\right\rbrace \right|\\
				&+\frac{\rho^{1-\frac{\gamma}{k}}}{k\Gamma_{k}(\gamma)}\left|\left\lbrace \Psi(x_{2},\alpha(x_{2}))-\Psi(x_{1},\alpha(x_{1}))\right\rbrace \int_{1}^{x_{1}}\frac{t^{\rho-1}}{\left(x_{1}^{\rho}-t^{\rho} \right)^{1-\frac{\gamma}{k}}}G(t,\alpha(t))dt\right|.
		\end{align*}
		Moreover,
		\begin{align*}
			& 
			\left|F(x_{2},\alpha(x_{2}))-F(x_{1},\alpha(x_{1}))\right|\\
			&\leq\left|F(x_{2},\alpha(x_{2}))-F(x_{1},\alpha(x_{2}))\right|+\left|F(x_{1},\alpha(x_{2}))-F(x_{1},\alpha(x_{1}))\right|\\
			&\leq\left|F(x_{2},\alpha(x_{2}))-F(x_{1},\alpha(x_{2}))\right|+c_{1}\left|\alpha(x_{2})-\alpha(x_{1})\right|\\
			&\leq\mu_{r_{0}}(F, \delta)+c_{1}\mu(\alpha,\delta).
		\end{align*}
			where,
		\[
		\mu_{r_{0}}(F, \delta)=\sup\left\{\begin{array}{ccl}
		\left|F(x_{2},\alpha(x_{2}))-F(x_{1},\alpha(x_{2}))\right|:\left| x_{2}-x_{1}\right|\leq \delta; x_{1},x_{2} \in \Bar{\mathds{J}}, \parallel \alpha \parallel \leq r_{0}
		\end{array}\right\}.
		\]
		and
		\begin{align*}
			&\left|\Psi(x_{2},\alpha(x_{2})) \left\lbrace \int_{1}^{x_{2}}\frac{t^{\rho-1}}{\left(x_{2}^{\rho}-t^{\rho} \right)^{1-\frac{\gamma}{k}}}G(t,\alpha(t))dt-\int_{1}^{x_{1}}\frac{t^{\rho-1}}{\left(x_{1}^{\rho}-t^{\rho} \right)^{1-\frac{\gamma}{k}}}G(t,\alpha(t))dt\right\rbrace \right|\\
			&\leq\left|\Psi(x_{2},\alpha(x_{2}))\right| \left| \int_{1}^{x_{2}}\frac{t^{\rho-1}}{\left(x_{2}^{\rho}-t^{\rho} \right)^{1-\frac{\gamma}{k}}}G(t,\alpha(t))dt-\int_{1}^{x_{1}}\frac{t^{\rho-1}}{\left(x_{1}^{\rho}-t^{\rho} \right)^{1-\frac{\gamma}{k}}}G(t,\alpha(t))dt \right|\\
			&\leq c_{2}r_{0} \left| \int_{1}^{x_{2}}\frac{t^{\rho-1}}{\left(x_{2}^{\rho}-t^{\rho} \right)^{1-\frac{\gamma}{k}}}G(t,\alpha(t))dt-\int_{1}^{x_{1}}\frac{t^{\rho-1}}{\left(x_{1}^{\rho}-t^{\rho} \right)^{1-\frac{\gamma}{k}}}G(t,\alpha(t))dt \right|\\
			&\leq c_{2}r_{0} \left| \int_{1}^{x_{2}}\frac{t^{\rho-1}}{\left(x_{2}^{\rho}-t^{\rho} \right)^{1-\frac{\gamma}{k}}}G(t,\alpha(t))dt-\int_{1}^{x_{2}}\frac{t^{\rho-1}}{\left(x_{1}^{\rho}-t^{\rho} \right)^{1-\frac{\gamma}{k}}}G(t,\alpha(t))dt \right|\\
			&+ c_{2}r_{0} \left| \int_{1}^{x_{2}}\frac{t^{\rho-1}}{\left(x_{1}^{\rho}-t^{\rho} \right)^{1-\frac{\gamma}{k}}}G(t,\alpha(t))dt-\int_{1}^{x_{1}}\frac{t^{\rho-1}}{\left(x_{1}^{\rho}-t^{\rho} \right)^{1-\frac{\gamma}{k}}}G(t,\alpha(t))dt \right|\\
			&\leq c_{2}r_{0} \left| \int_{1}^{x_{2}}\left\lbrace \frac{t^{\rho-1}}{\left(x_{2}^{\rho}-t^{\rho} \right)^{1-\frac{\gamma}{k}}}-\frac{t^{\rho-1}}{\left(x_{1}^{\rho}-t^{\rho} \right)^{1-\frac{\gamma}{k}}}\right\rbrace G(t,\alpha(t))dt \right|\\
			&+ c_{2}r_{0} \left| \int_{x_{2}}^{x_{1}}\frac{t^{\rho-1}}{\left(x_{1}^{\rho}-t^{\rho} \right)^{1-\frac{\gamma}{k}}}G(t,\alpha(t))dt \right|\\
			&\leq c_{2}c_{3}r_{0}^{2}\frac{k}{\rho\gamma}\left[2\left(x_{1}^{\rho}-x_{2}^{\rho} \right)^{\frac{\gamma}{k}}+\left(x_{2}^{\rho}-1 \right)^{\frac{\gamma}{k}} - \left(x_{1}^{\rho}-1 \right)^{\frac{\gamma}{k}} \right].
		\end{align*}
		also,
		\begin{align*}
			&\left|\left\lbrace \Psi(x_{2},\alpha(x_{2}))-\Psi(x_{1},\alpha(x_{1}))\right\rbrace \int_{1}^{x_{1}}\frac{t^{\rho-1}}{\left(x_{1}^{\rho}-t^{\rho} \right)^{1-\frac{\gamma}{k}}}G(t,\alpha(t))dt\right|\\
			&\leq\left| \Psi(x_{2},\alpha(x_{2}))-\Psi(x_{1},\alpha(x_{1}))\right| \int_{1}^{x_{1}}\frac{t^{\rho-1}}{\left(x_{1}^{\rho}-t^{\rho} \right)^{1-\frac{\gamma}{k}}}\left|G(t,\alpha(t))\right|dt\\
			&\leq\left\lbrace\left| \Psi(x_{2},\alpha(x_{2}))-\Psi(x_{2},\alpha(x_{1}))\right|+\left| \Psi(x_{2},\alpha(x_{1}))-\Psi(x_{1},\alpha(x_{1}))\right|\right\rbrace \int_{1}^{x_{1}}\frac{t^{\rho-1}}{\left(x_{1}^{\rho}-t^{\rho} \right)^{1-\frac{\gamma}{k}}}\left|G(t,\alpha(t))\right|dt\\
			&\leq\left\lbrace c_{2}\left| \alpha(x_{2})-\alpha(x_{1})\right|+\Psi_{r_{0}}( \delta)\right\rbrace \frac{c_{3}  r_{0}k}{\gamma \rho}\left( \mathbb{T}^{\rho}-1\right)^{\frac{\gamma}{k}}.
		\end{align*}
		where
		\[
		\Psi_{r_{0}}( \delta)=\sup\left\{\begin{array}{ccl}
			\left|\Psi(x_{2},\alpha(x_{1}))-\Psi(x_{1},\alpha(x_{1}))\right|:\left| x_{2}-x_{1}\right|\leq \delta; x_{1},x_{2} \in \Bar{\mathds{J}}, \parallel \alpha \parallel \leq r_{0}
		\end{array}\right\}.
		\].
	Therefore,
	\begin{align*}
		&\left| (\mathscr{D} \alpha)(x_{2})-  (\mathscr{D} \alpha)(x_{1})\right|\\
		&\leq 
		\mu_{r_{0}}(F, \delta)+c_{1}\mu(\alpha,\delta)+\frac{c_{2}c_{3}r_{0}^{2}\rho^{-\frac{\gamma}{k}}}{\gamma\Gamma_{k}(\gamma)}\left[2\left(x_{1}^{\rho}-x_{2}^{\rho}\right)^{\frac{\gamma}{k}}+\left(x_{2}^{\rho}-1 \right)^{\frac{\gamma}{k}} - \left(x_{1}^{\rho}-1 \right)^{\frac{\gamma}{k}} \right]\\
		&+\frac{c_{3}  r_{0}\rho^{-\frac{\gamma}{k}}}{\gamma\Gamma_{k}(\gamma)}\left\lbrace c_{2}\mu(\alpha,\delta)+\Psi_{r_{0}}( \delta)\right\rbrace \left( \mathbb{T}^{\rho}-1\right)^{\frac{\gamma}{k}}.
	\end{align*}
	i.e.
	\begin{align*}
		\mu(\mathscr{D} \alpha,\delta)
		&\leq 
		\mu_{r_{0}}(F, \delta)+c_{1}\mu(\alpha,\delta)+\frac{c_{2}c_{3}r_{0}^{2}\rho^{-\frac{\gamma}{k}}}{\gamma\Gamma_{k}(\gamma)}\left[2\left(x_{1}^{\rho}-x_{2}^{\rho}\right)^{\frac{\gamma}{k}}+\left(x_{2}^{\rho}-1 \right)^{\frac{\gamma}{k}} - \left(x_{1}^{\rho}-1 \right)^{\frac{\gamma}{k}} \right]\\
		&+\frac{c_{3}  r_{0}\rho^{-\frac{\gamma}{k}}}{\gamma\Gamma_{k}(\gamma)}\left\lbrace c_{2}\mu(\alpha,\delta)+\Psi_{r_{0}}( \delta)\right\rbrace \left( \mathbb{T}^{\rho}-1\right)^{\frac{\gamma}{k}}.
	\end{align*}
	which yields
	\begin{align*}
		\mu(\mathscr{D} \varpi_{\alpha},\delta)
		&\leq 
		\mu_{r_{0}}(F, \delta)+c_{1}\mu(\varpi_{\alpha},\delta)+\frac{c_{2}c_{3}r_{0}^{2}\rho^{-\frac{\gamma}{k}}}{\gamma\Gamma_{k}(\gamma)}\left[2\left(x_{1}^{\rho}-x_{2}^{\rho}\right)^{\frac{\gamma}{k}}+\left(x_{2}^{\rho}-1 \right)^{\frac{\gamma}{k}} - \left(x_{1}^{\rho}-1 \right)^{\frac{\gamma}{k}} \right]\\
		&+\frac{c_{3}  r_{0}\rho^{-\frac{\gamma}{k}}}{\gamma\Gamma_{k}(\gamma)}\left\lbrace c_{2}\mu(\varpi_{\alpha},\delta)+\Psi_{r_{0}}( \delta)\right\rbrace \left( \mathbb{T}^{\rho}-1\right)^{\frac{\gamma}{k}}.
	\end{align*}
	As $\;\delta\; \rightarrow \;0,$ then $x_{2} \rightarrow x_{1}$ and by the uniform continuity of $F$ and $\Psi$ such  that \\ $\mu_{r_{0}}(F, \delta)\rightarrow\; 0$ and $\Psi_{r_{0}}( \delta)\rightarrow\; 0,$ as $\;\delta\; \rightarrow \;0$ we obtain
	\begin{align*}
		\mu_{0}(\mathscr{D} \varpi_{\alpha})
		&\leq \left[ c_{1}
		+\frac{c_{2}c_{3}  r_{0}\rho^{-\frac{\gamma}{k}}}{\gamma\Gamma_{k}(\gamma)}  \left( \mathbb{T}^{\rho}-1\right)^{\frac{\gamma}{k}}\right]\mu_{0}(\varpi_{\alpha}) .
	\end{align*}
	Similarly we obtain
	\begin{align*}
		\mu_{0}(\mathscr{\bar{D}} \varpi_{\beta})
		&\leq \left[ \hat{c}_{1}
		+\frac{\hat{c}_{2}\hat{c}_{3}  r_{0}\rho^{-\frac{\gamma}{k}}}{\gamma\Gamma_{k}(\gamma)}  \left( \mathbb{T}^{\rho}-1\right)^{\frac{\gamma}{k}}\right]\mu_{0}(\varpi_{\beta}) .
	\end{align*}
		\par Setting 
		\[
	\varepsilon=max\left\lbrace c_{1}
	+\frac{c_{2}c_{3}  r_{0}\rho^{-\frac{\gamma}{k}}}{\gamma\Gamma_{k}(\gamma)}  \left( \mathbb{T}^{\rho}-1\right)^{\frac{\gamma}{k}}, \hat{c}_{1}
	+\frac{\hat{c}_{2}\hat{c}_{3}  r_{0}\rho^{-\frac{\gamma}{k}}}{\gamma\Gamma_{k}(\gamma)}  \left( \mathbb{T}^{\rho}-1\right)^{\frac{\gamma}{k}}\right\rbrace 	<1.
		\]
		Hence, 
		\begin{align*}
			\mu_{0}(\mathscr{D} \varpi_{\alpha}) + \mu_{0}(\mathscr{\bar{D}} \varpi_{\beta})
			&\leq \delta[\mu_{0}(\varpi_{\alpha}) +\mu_{0}(\varpi_{\beta})] .
		\end{align*}
		Hereby Corollary \ref{c1}, $\mathscr{D}$ and $\mathscr{\bar{D}}$ have fixed point in $\varpi_{\alpha}, \varpi_{\beta} \subseteq \mathscr{Q}_{r_{0}}.$ 
		\\Which implies the equations (\ref{eq75}) and (\ref{eq77}) have solutions in $\mathscr{B}.$
	\end{proof}

	%**************************** Example *********************
	\begin{ill}\label{ex7}{\rm
			Take into consideration the following fractional integral equations:
			%Consider the $\mathbb{FIE}$ as follows:
			\begin{equation}\label{eq76}
				\alpha(x)=\frac{\left|\alpha(x)\right|}{6} +\frac{9\left|\alpha(x)\right|}{\Gamma_{1/3}(2/3)}\int_{1}^{x}\frac{\left(x^{1/3}-t^{1/3} \right)}{t^{2/3}}\frac{\alpha(x)}{3+\log{x}}dt 
			\end{equation}
			and
			\begin{equation}\label{eq78}
				\beta(x)=\frac{\left|\beta(x)\right|}{6} +\frac{9\left|\beta(x)\right|}{\Gamma_{1/3}(2/3)}\int_{1}^{x}\frac{\left(x^{1/3}-t^{1/3} \right)}{t^{2/3}}\frac{\beta(x)}{2+x}dt 
			\end{equation}
			for $x\in [1,3]=\Bar{\mathds{J}}.$}
		\end{ill}
	\noindent  Here, \[ \gamma=2/3,\; k=1/3,\; \rho= 1/3\]\\
	and
	\[
	F(x,\alpha(x))=\hat{F}(x,\beta(x))=\frac{\left|\alpha(x)\right|}{6},
	\]
	\[
	\Psi(x,\alpha(x))=\hat{\Psi}(x,\beta(x))=\left|\alpha(x)\right|,
	\]
	\[
	G(t,\alpha(t))=\frac{\alpha(x)}{3+\log{x}},
	\]
	\[
	\hat{G}(t,\beta(t))=\frac{\beta(x)}{2+x}.
	\]
	Therefore
	\[
	\left| F(x,\alpha(x))-F(x,\bar{\alpha}(x))\right|\leq \frac{1}{6} \left| \alpha(x)-\bar{\alpha}(x)\right|
	\]
	\[
	\left| \hat{F}(x,\beta(x))-\hat{F}(x,\bar{\beta}(x))\right|\leq \frac{1}{6} \left| \beta(x)-\bar{\beta}(x)\right|
	\]
	\[
	\left| \Psi(x,\alpha(x))-\Psi(x,\bar{\alpha}(x))\right|\leq  \left| \alpha(x)-\bar{\alpha}(x)\right|
	\]
	\[
	\left| \hat{\Psi}(x,\beta(x))-\hat{\Psi}(x,\bar{\beta}(x))\right|\leq  \left| \beta(x)-\bar{\beta}(x)\right|
	\]
	\[
	\left| G(x,\alpha(x))-G(x,\bar{\alpha}(x)) \right|\leq \frac{1}{3}  \left| \alpha(x)-\bar{\alpha}(x)\right|
	\]
	and
	\[
	\left|  \hat{G}(x,\beta(x))-\hat{G}(x,\bar{\beta}(x)) \right|\leq \frac{1}{3}  \left| \beta(x)-\bar{\beta}(x)\right|
	\]
	Also, 
	\[
	F(x,0)=\hat{F}(x,0)=\Psi(x,0)=\hat{\Psi}(x,0)=G(x,0)=\hat{G}(x,0)=0.
	\]
Then, 
\[
{c}_{1}=\hat{c}_{1}=\frac{1}{6},\;{c}_{2}=\hat{c}_{2}=1,\;{c}_{3}=\hat{c}_{3}=\frac{1}{3} .
\]
Substituting these vaues in the first inequality of assumption (C2), we get
\begin{align*}
	&\frac{1}{6}+
	\frac{1.\frac{1}{3}  r_{0}(\frac{1}{3})^{-2}}{\frac{2}{3} \Gamma_{1/3}(2/3)}\left( 3^{1/3}-1\right)^{2}< 1\\
	&\Rightarrow\frac{1}{6}+
	\frac{9 r_{0}}{2 \Gamma_{1/3}(2/3)}\left( 3^{1/3}-1\right)^{2}< 1\\
	&\Rightarrow\frac{1}{6}+
	\frac{9 r_{0}}{2 (2.4047)}\left(0.5358\right)< 1\\
	& \Rightarrow r_{0}< 0.83.
\end{align*}
Similarly, from the second inequality of assumption (C2), we get
\begin{equation*}
	r_{0}< 0.83.
\end{equation*}
	However, the assumption $(C2)$ is also satisfied for $r_{0}=0.83.$

	Consequently, we have achieved all of the assumptions from $(C1)$ to $(C2)$ in the Theorem \ref{T74} .\\
	From Theorem \ref{T74} , we can say that the equations (\ref{eq76}) and (\ref{eq78}) have solutions in $\mathscr{B}=\Bar{\mathds{C}}(\Bar{\mathds{J}}).$
	
	%--------------------------------------------------	
	%---------------------------------------------------
	\section{Conclusion}
	Using generalized DFPT along with $\mathbb{H}$-class mappings, we can derive new results concerning fixed point theorems. This approach shows that Darbo's fixed point theorem is effective in establishing important fixed point results, particularly in proving the existence of solutions for a system of  $(k,\rho)$-fractional Hilfer integral equations.

\section{ Data Availability Statement:}
No data were used to support this study.

\section{Competing Interest}
The authors declare that there is not any competing interest regarding the publication of this manuscript.

\section{Author contributions} All authors contributed equally to the writing of this paper. All authors read and
approved the final manuscript.

\section{Funding}
This research received external funding.

	\section{Acknowledgement}
      %This work was sponsored by the Department of Biotechnology (DBT) of the Government of India's Star College Scheme.
      %The authors are grateful to the referee, who carefully read our manuscript and offered useful suggestions. 
      The authors acknowledge the financial support from the Department of Biotechnology (DBT), Government of India, under the Star College Scheme [HRD-11011/17/2023-HRD-DBT].

\end{document}